\RequirePackage[english]{babel}
\documentclass[a4paper,12pt,reqno]{amsart}

\usepackage{german}
\usepackage[english]{babel}
\usepackage[utf8]{inputenc}
\usepackage{graphicx,color}
\usepackage[a4paper]{geometry}

\theoremstyle{plain}
\newtheorem{thm}{Theorem}[section]
\newtheorem*{thm*}{Theorem}
\newtheorem{lem}[thm]{Lemma}
\newtheorem{prop}[thm]{Proposition}
\newtheorem{cor}[thm]{Corollary}

\theoremstyle{definition}

\newtheorem{conj}[thm]{Conjecture}

\theoremstyle{remark}

\newcommand{\R}{\mathbb{R}}
\newcommand{\RR}{\mathbb{R}}
\newcommand{\Ss}{\mathcal{S}} 

\DeclareMathOperator{\supp}{{supp}}

\newcommand{\ch}{{H}} 
\newcommand{\cb}{{B}} 
\newcommand{\cp}{{\mathbb{P}}} 
\newcommand{\ct}{{T}} 

\newcommand{\comb}{\mathbb{B}} 
\newcommand{\comt}{\mathbb{T}} 

\newcommand{\rg}{G} 
\newcommand{\dg}{D}
\newcommand{\arc}{\mathcal{A}}

\newcommand{\ci}{\widetilde}
\newcommand{\ws}{\alpha} 

\newcommand\SetOf[2]{\left\{#1\vphantom{#2}\,\right.\left|\,\vphantom{#1}#2\right\}}
\newcommand\smallSetOf[2]{\{#1\,|\,#2\}}
\newcommand\cP{{\mathcal P}}

\title[Optimal realisations and tight-spans]{Optimal realisations of two-dimensional, totally-decomposable metrics}

\author[Herrmann \and Koolen \and Lesser \and Moulton \and Wu]{Sven Herrmann \and Jack H. Koolen \and  Alice Lesser \and  Vincent Moulton \and Taoyang Wu }

\address{Sven Herrmann, School of Computing Sciences, University of East Anglia, Norwich, NR4 7TJ, UK}
\email{mail@svenherrmann.net}
\address{Jack H. Koolen, Wen-Tsun Wu Key Laboratory of CAS, School of Mathematical Sciences,   
University of Science and Technology of China (USTC)
 China}
\email{koolen@ustc.edu.cn}
\address{Alice Lesser, Department of Mathematics, Uppsala University, Box 480, 751 06 Uppsala, Sweden}
\email{alice@lesser.se}
\address{Vincent Moulton, School of Computing Sciences, University of East Anglia, Norwich, NR4 7TJ, UK}
\email{vincent.moulton@cmp.uea.ac.uk}
\address{Taoyang Wu, School of Computing Sciences, University of East Anglia, Norwich, NR4 7TJ, UK, and Department of Mathematics, National University of
Singapore, 10 Lower Kent Ridge Road, 119076, Singapore 
}
\email{taoyang.wu@gmail.com}

\begin{document}

\begin{abstract}
A realisation of a metric $d$ on a finite set $X$
is a weighted graph $(G,w)$ whose vertex 
set contains $X$ such that the 
shortest-path distance between 
elements of $X$ considered as vertices in $G$ is equal to $d$.
Such a realisation $(G,w)$ is called optimal 
if the sum of its edge weights is minimal 
over all such realisations. Optimal realisations
always exist, although it is NP-hard to 
compute them in general, and they have applications in areas
such as phylogenetics, electrical networks  
and internet tomography. In
[{\em Adv. in Math.} {\bf 53}, 1984, 321-402]
A.~Dress showed that the optimal realisations 
of a metric $d$ are closely related to a certain 
polytopal complex that can be canonically 
associated to $d$ called its tight-span. 
Moreover, he conjectured that the 
(weighted) graph consisting of the zero- and one-dimensional 
faces of the tight-span of $d$ 
must always contain an optimal realisation as a
homeomorphic subgraph. 
In this paper, we prove that this 
conjecture does indeed hold for a certain class of
metrics, namely the class of totally"=decomposable metrics 
whose tight-span has dimension two. 
As a corollary, it follows that the minimum Manhattan network problem is a special case of finding optimal realisations of two-dimensional totally"=decomposable metrics.\\ 


\end{abstract}

\maketitle

\section{Introduction}

Let $(X,d)$ be a finite metric space, that is, a finite 
set $X$, $|X| \ge 2$, together with a metric $d$ (i.e., 
a symmetric map
$d:X \times X \to \mathbb{R}_{\ge 0}$ that vanishes precisely on 
the diagonal and that satisfies the triangle inequality).
A \emph{realisation} $(G,w)$ of $(X,d)$ consists of 
a graph $G=(V(G),E(G))$ with $X$ a subset of the
vertex set $V(G)$ of $G$, 
together with a weighting $w:E(G) \to \mathbb{R}_{>0}$
on the edge set $E(G)$ of $G$  
such that for all $x,y \in X$ the length of  
any shortest path in $(G,w)$ between $x$ and $y$ equals $d(x,y)$.
A realisation $(G,w)$ of $d$ is called \emph{optimal} if
$\sum_{e \in E(G)}w(e)$  is
minimal amongst all realisations of $(X,d)$. 

Realising metrics by graphs has applications in
fields such as phylogenetics, 
electrical networks 
and internet tomography. 
Optimal realisations were introduced by
Hakimi and Yau~\cite{HY65} who also gave a polynomial algorithm
for their computation in the special case where the metric space 
has a (necessarily unique) optimal realisation that is a  tree. 
Every finite metric space has an 
optimal realisation \cite{Dress84,Imrich84}, 
although they are 
not necessarily unique \cite{Althofer88,Dress84}.
In general, it is NP-hard to compute 
optimal realisations~\cite{Althofer88,Winkler88},
although recently some progress has been 
made in deriving heuristics for their computation
\cite{HV07,HV08}.

In \cite{Dress84}, Dress pointed out an intriguing
connection between optimal realisations and tight-spans,
which we now recall. 
The {\em tight-span} $\ct(d)$ of the metric
space $(X,d)$ \cite{Dress84,Isbell}
is the set of all 
minimal elements (with respect to the product order) of 
the polyhedron
\[
\cp(d) := \{f \in \mathbb{R}^X : 
f(x)+f(y)\geq d(x,y) \text{ for all } x,y \in X\}\,.
\]
Note that, in particular, $T(d)$ consists of the union of the 
bounded faces of $P(d)$. Moreover, 
the map $d_\infty$, given by $d_\infty(f,g)=\sup_{x\in X} |f(x)-g(x)|$ 
for all $f,g\in\cp(d)$, is a metric on $T(d)$
and the {\em Kuratowski map}
\[
\kappa:X \rightarrow \ct(d): x \to h_x; \,\,\,\,\, h_x(y):=d(x,y), 
\mbox{ for all } x \in X,
\]
gives an isometric embedding of $(X,d)$ into 
$(T(d),d_\infty)$; that is, $\kappa$ is 
injective and preserves distances. 

In \cite[Theorem~5]{Dress84}, Dress
showed that the 
(necessarily finite and connected) weighted graph $G_d$ 
consisting of the zero- and one-dimensional faces of $T(d)$ 
and weighting $w_\infty$ defined by 
$w_\infty(\{f,g\}):=d_\infty(f,g)$, $f,g$ zero-dimensional faces of $T(d)$,  
is homeomorphic to
a realisation of $d$ 
(see Section~\ref{sec:prelim} for relevant definitions). 
Moreover, he showed that if 
$(G,w)$ is any optimal realisation of $(X,d)$,
then there exists a certain map $\psi:V(G) \to T(d)$
of the vertices of $G$ into $T(d)$ \cite[Theorem 5]{Dress84}
(see also Theorem~\ref{thm:inj:geometric:graph} below).
This led him to suspect that  every optimal realisation of $(X,d)$ 
is homeomorphic to a subgraph of  $(G_d,w_\infty)$.
Even though this conjecture was disproven 
by Alth\"ofer \cite{Althofer88}, the following 
related  conjecture is still open:

\begin{conj}[cf. (3.20) in \cite{Dress84}]\label{conj:dress}
Let $(X,d)$ be a finite metric space.  Then 
there exists an optimal realisation of $(X,d)$ 
that is homeomorphic to a subgraph of $(G_d,w_\infty)$.
\end{conj}

Apart from having an intrinsic mathematical interest, if
this conjecture were true, it could provide
new strategies for computing optimal
realisations, as it would provide a ``search space''
(albeit a rather large one in general) in 
which to systematically search for optimal realisation \cite{HMS13}.

Conjecture \ref{conj:dress} is known to hold
for metrics $d$ that can be realised by a tree since
in this case  $(G_d,w_\infty)$ is precisely 
the tree that realises $d$ uniquely~\cite{Imrich84}. 
In this paper, we show that it also 
holds for a certain 
class of metrics that generalise tree metrics.
More specifically, for a finite 
metric space $(X,d)$ as above, define,  
for any four elements $x,y,u,v \in X$, 
\[
\beta(x,y;u,v) := \max\{d(x,u)+d(y,v),d(x,v)+d(y,u)\} - d(x,y) -d(u,v)
\]
and put $\alpha(x,y;u,v) := \max(\beta(x,y;u,v),0)$.
The metric $d$ is called {\em totally"=decomposable} if for all 
$t,x,y,u,v\in X$ the inequality 
$\beta(x,y;u,v)\leq \alpha(x,t;u,v)+ \alpha(x,y;u,t)$ holds \cite{BD92}.
Such metrics are commonly used to understand genetic data in phylogenetic analysis. 
Defining the \emph{dimension} of $d$ 
to be the dimension of $T(d)$
(regarded as a subset of $\mathbb{R}^X$), we shall 
prove the following result.

\begin{thm}\label{2compat}
Let $(X,d)$ be a totally"=decomposable finite metric space with dimension two.
Then there exists an optimal realisation of $(X,d)$ that is
homeomorphic to a subgraph of  $(G_d,w_\infty)$.
\end{thm}

In fact this immediately follows from a somewhat 
stronger theorem that we shall prove
(Theorem~\ref{thm:min-path-sat}), which shows that 
a certain special type of optimal 
realisation of a two-dimensional, totally-decomposable 
metric $d$ can be found as a homeomorphic subgraph
of $(G_d,w_\infty)$. 
 Note also that Theorem~\ref{2compat} implies that the optimal realisation problem for $l_1$-planar metrics is equivalent to the Minimum Manhattan Network (MMN) problem; since the MMN problem is NP-hard~\cite{chin2011min}, the optimal realisation problem for two-dimensional metrics is also NP-hard (see \cite[Section 5]{HMS13} for more details and some algorithmic consequences).

Our proof of Theorem~\ref{2compat} heavily 
relies on the two-dimensionality of the tight-span, and  
we do not know how to extend our arguments to 
totally-decomposable metrics. 
Even so, it might be of interest to try and extend our
result to two-dimensional metrics in general, especially  
as a great deal is known concerning the structure of 
their tight-spans (e.g. \cite{Hirai09,Karzanov98}). Indeed, 
our proof of Theorem~\ref{2compat} relies on a close relationship between 
tight-spans and so-called Buneman  or
median complexes, and so results concerning median complexes
and folder complexes \cite{Chepoi00}
could potentially help yield a more general
result for two-dimensional metrics.

The remainder of this paper is organised as follows.  We recall some definitions and results in Section~\ref{sec:prelim}. 
We will then present a theorem about embeddings  of realisations into the Buneman complex in Section~\ref{sec:split:flow} which uses
the new notions of split-flow digraphs and split potentials.  
Finally, we establish our main result in Section~\ref{sec:proof}, from which Theorem~\ref{2compat} follows.

\section{Preliminaries and previous results}\label{sec:prelim}

In this section, we will state the known definitions 
and results that are used in the rest of the paper. 

\subsection{Graphs}\label{ssec:graphs}

A \emph{weighted graph} $(G,w)$ is a 
graph $G$ with vertex set $V(G)$ and 
edge set $E(G)\subseteq \binom{V(G)}2$ together 
with a weight function 
$w: E(G) \to \mathbb{R}_{>0}$ that assigns a 
positive weight or {\em length} to each edge. 
A weighted graph $(G',w')$ is a \emph{subgraph} 
of $(G,w)$ if 
$V(G')\subseteq V(G)$, $E(G')\subseteq 
\smallSetOf{e\in E(G)}{e\subseteq V(G')}$ and $w'=w|_{E(G')}$. 
The \emph{length} of $(G,w)$ is $l(G,w):=\sum_{e\in E(G)}w(e)$. 
A {\em path} 
$P$ from $u$ to $v$ in $G$ is a sequence $u=v_0,v_1,\dots,v_{k-1},v_k=v$ 
of distinct vertices in $G$ such that $\{v_{i-1},v_{i}\}\in E(G)$ for all 
$1 \le i \le k$. Note that $u$ and $v$ will be referred to as  the {\em ends}  of $P$ and the {\em length} of $P$ is defined as
$w(P):=\sum_{i=1}^k w(\{v_{i-1},v_i\})$. 
It is easily observed that for any $W\subseteq V(G)$ 
the map setting $d_{(G,w)}(u,v)$ to be the length of a 
shortest path between $u$ and $v$ defines 
a metric space $(W,d_{(G,w)})$.

We {\em suppress} a vertex of degree two in a weighted graph  
when we remove it and replace its two incident edges by 
a single edge whose length is equal to the sum of their lengths.
Given two weighted graphs $(G_1,w_1)$ and $(G_2,w_2)$, 
they are {\em isomorphic} if there exists an isomorphism 
between $G_1$ and $G_2$ that also preserves the length 
of each edge; they are {\em homeomorphic} if there exist 
two isomorphic weighted graphs $(G'_1,w'_1)$ and 
$(G'_1,w'_2)$ such that  $(G'_i,w'_i)$ ($i=1,2$) can be  
obtained from $(G_i,w_i)$ by suppressing a sequence 
of degree two vertices. 

As mentioned in the introduction, a weighted graph $(G,w)$ 
with $X\subseteq V(G)$ 
and $d=d_{(G,w)}$ is called a \emph{realisation} 
of the metric space $(X,d)$. The elements in $V(G)\setminus X$ 
are called \emph{auxiliary vertices} of the realisation, 
and throughout this paper we will use the convention 
that all auxiliary vertices of degree two are 
suppressed. 

\subsection{Optimal realisations and geodesics}

We now recall some well-known observations 
concerning optimal realisations.

\begin{lem}[Lemma~2.1 in \cite{Althofer88}]
\label{lem:suff:opt:realisation}
Let $(G,w)$ be an optimal realisation of a finite metric space $(X,d)$. Then:
\begin{enumerate}
\item For any edge $e\in E(G)$, there exist two 
elements $x,x'\in X$ such that $e$ belongs to all shortest 
paths between $x$ and $x'$.
\item \label{lem:suff:opt:realisation:two}For any 
two edges in $E(G)$ that share a common vertex, 
there exists a shortest path between two elements 
of $X$ that contains these edges. 
\end{enumerate}
\end{lem}

As an immediate consequence of Lemma~\ref{lem:suff:opt:realisation}(2),   we have:

\begin{cor}
\label{cor:triangle:free}
Let $(G,w)$ be an optimal realisation of a 
finite metric space $(X,d)$. Then $G$ is triangle-free.
\end{cor}


As mentioned in the introduction, not all 
optimal realisations are homeomorphic to 
subgraphs of the tight-span. 
However, we will now present some properties of
optimal realisations that will guarantee this property. 
If $(G,w)$ is a weighted graph and $A\subseteq V(G)$, 
we denote by $\Gamma(G,w;A)$ the set of all pairwise distinct 
shortest paths in $G$ connecting elements of $A$. 
By~\cite[Proposition~7.1]{KLMW}, for each metric space $(X,d)$ there exists 
a \emph{path-saturated} optimal realisation $(G,w)$ of $(X,d)$ such that  
$|\Gamma(G,w;X)|\geq |\Gamma(G',w';X)|$ holds 
for all optimal realisations $(G',w')$ of $(X,d)$. If in addition the number of vertices $V(G)$ is 
minimal among all path-saturated realisations of $(X,d)$, 
then $(G,w)$ is called a 
\emph{minimal path-saturated realisation} of $(X,d)$.

Now, if $(X,d)$ is a (not necessarily finite) 
metric space, a function $\gamma:[0,1]\to X$ is 
called a \emph{geodesic} in $(X,d)$ if for all 
$a<b<c\in [0,1]$ one has 
$d(\gamma(a),\gamma(c))=d(\gamma(a),\gamma(b))+d(\gamma(b),\gamma(c))$. 
Note that this implies that $\gamma$ is continuous.
A map $\psi:X \to X'$ between two arbitrary metric spaces 
$(X,d)$ and $(X',d')$ is called  {\em non-expansive}, if 
$d'\big(\psi(x_1),\psi(x_2)\big) \le d(x_1,x_2)$ 
for all $x_1,x_2 \in X$. If $\psi^{-1}$ exists and 
is non-expansive, too, $\psi$ is an \emph{isometry} 
and $(X,d)$ and $(X',d')$ are said to be \emph{isometric}.

For a weighted graph $(G,w)$ we denote by $||(G,w)||$ 
its {\em geometric realisation},
that is, the metric space obtained by regarding each 
edge $e\in E(G)$ as a real interval of length $w(e)$ 
and gluing them together at
the vertices of $G$ (see, e.g., Daverman and 
Sher \cite[p. 547]{Dav-She-02} for details of 
this construction). For $X\subset V(G)$ a 
function $\gamma:[0,1]\to||(G,w)||$ is called 
an \emph{$X$-geodesic} if it is a geodesic 
between two points of $X$ interpreted as points in $||(G,w)||$.

The following theorem gives us a way to relate 
the geometric realisation of an
optimal realisation of a metric with its tight-span. The 
first part is due to Dress \cite[Theorem~5]{Dress84}, 
and the second part is given in~\cite[Proposition~7.1]{KLMW}.

\begin{thm}
\label{thm:inj:geometric:graph}
Let $(G,w)$ be an optimal realisation of a finite 
metric space $(X,d)$. Then there exists a 
non-expansive map $\psi$ from $||(G,w)||$ 
to $(\ct(d),d_\infty)$ such that $\psi(x)=\kappa(x)$ 
for all $x\in X$. If, in addition, $(G,w)$ is 
path-saturated, then $\psi$ is injective.
\end{thm}

\subsection{Splits and total-decomposability}

A {\em split} $S=\{A,B\}$ of a finite set $X$ is a 
bipartition of $X$, that is $A\cup B=X$ 
and $A\cap B=\emptyset$. A {\em weighted split system} 
$(\Ss,\ws)$ on~$X$ is a pair consisting of a set 
$\Ss$ of splits of $X$, and a weight function 
$\ws: \Ss\rightarrow \R_{>0}$. For all $x,y\in X$, we set 
$\Ss(x,y)=\smallSetOf{\{A,B\}\in\Ss}{x\in A,y\in B\text{ or }x\in B,y\in A}$ 
and define
\[
d_{(\Ss,\ws)}(x,y)=\sum_{S\in \Ss(x,y)}{\ws(S)}\,.
\]
If $\Ss(x,y)\not=\emptyset$ for all distinct $x,y\in X$, 
the pair $(X,d_{(\Ss,\ws)})$ becomes a finite metric space.

Two splits $\{A,B\}$ and $\{A',B'\}$ of $X$ are 
called {\em incompatible} if none of the four 
intersections $A\cap A'$, $A\cap B'$, $B\cap A'$ and  $B\cap B'$ is empty.
A weighted split system $(\Ss,\ws)$ is called 
(1) {\em two-compatible} if $\Ss$ does not contain three 
pairwise incompatible splits, (2) {\em weakly compatible} 
if for any three splits $S_1,S_2,S_3$ in $\Ss$, there 
exist $A_i \in S_i$, for each $i\in \{1,2,3\}$, such
that $A_1\cap A_2 \cap A_3=\emptyset$, and (3) 
{\em octahedral-free} 
if there exists no partition $X = X_1\cup \cdots \cup  X_6$ 
of $X$ into six non-empty disjoint subsets $X_i$, $1\leq i \leq 6$, 
such that each one of the following four splits:
\begin{eqnarray*}
S_1 =\{X_1\cup X_2\cup X_3, X_4\cup X_5 \cup X_6\}, & 
S_2 =\{X_2\cup X_3\cup X_4, X_5\cup X_6 \cup X_1\}, \\
S_3 =\{X_3\cup X_4\cup X_5, X_6\cup X_1 \cup X_2\}, &
S_4 =\{X_1\cup X_3\cup X_5, X_2\cup X_4 \cup X_6\}
\end{eqnarray*}
belongs to $\Ss$. Note that it is easily seen
that two-compatible split systems are
octahedral-free.

It can be shown
(cf. \cite{DHM01}) that a metric space $(X,d)$ is 
totally-decomposable if and only if there exists a weakly 
compatible weighted split system $(\Ss,\ws)$ 
on $X$ such that $d=d_{(\Ss,\ws)}$. 
Furthermore, if
$(X,d)$ is totally-decomposable, then $d$ 
has dimension at most two if and only if $(\Ss,\ws)$ is
two-compatible (cf. \cite{DHM01}). From now on 
we will call two-dimensional, totally-decomposable
metric spaces \emph{two-decomposable}.

\subsection{Polytopal complexes}\label{ssec:polyhedron}

We now recall some definitions about polytopes\linebreak 
(see~\cite{Ziegler} for further details). 
A {\em polyhedron} $P$ is  the intersection 
of a finite collection of halfspaces in a real 
vector space $\mathbb{V}$, and a {\em polytope} is a bounded 
polyhedron. For any linear functional $L:\mathbb{V} \to \RR$ 
the set $F=\smallSetOf{x\in P}{L(x)=\max_{y\in P}{L(y)}}$ 
is called a \emph{face} of $P$, as is the empty set. 
The zero- and one-dimensional faces are 
called {\em vertices} and {\em edges} of~$P$ and 
they naturally gives rise to a \emph{graph} of the 
polyhedron. A {\em cell complex} 
(or \emph{polytopal complex}) $\cp$ is a finite 
collection of polytopes (called \emph{cells}) 
such that each face of a member of $\cp$ is 
itself a member of $\cp$, and the intersection 
of two members of $\cp$ is a face of each. 
We denote the set of vertices of $\cp$ by $V(\cp)$. Two
cell complexes $\cp,\cp'$ are {\em isomorphic} 
if there exists an bijection $\pi$ (called 
\emph{cell-complex isomorphism}) between them 
such that for all $F,F'\in \cp$ the cell $F$ is a face 
of $F'$ if and only if $\pi(F)$ is a face of 
$\pi(F')$. 

For a finite metric space $(X,d)$, the set $P(d)$ 
defined in the introduction is obviously a 
polyhedron and it can be easily observed 
that $T(d)$ is the union of bounded faces 
of $P(d)$ (cf. \cite{Dress84})  and hence naturally carries the 
structure of a cell complex $\comt(d)$.

\subsection{The Buneman complex}
\label{subsect:BC}
The Buneman complex, also known as a median complex~\cite{Chepoi00},  is a 
cell complex that can be associated to 
any weighted split system, and 
that has proven useful in, for example, 
understanding the structure of the tight-span of a
totally-decomposable metric \cite{DHM02}.
Given a weighted split system $(\Ss,\ws)$ on $X$ 
we define its {\em support} to be the set 
$\supp(\ws):=\smallSetOf{A\subseteq X}{ 
\text{there exists }S\in\Ss\text{ with } A\in S}$. 
Consider the polytope (which is a hypercube)
\[
\ch(\Ss,\ws):=\SetOf{\mu\in \R^{\supp(\ws)}_{\geq 0}}
{\mu(A)+\mu(B)=\ws(S)\text{ for all }S\in \Ss}\,.
\]
Its subset
\[
\cb(\Ss,\ws):=\SetOf{\mu\in \ch(\ws)}{``\mu(A)\not = 0 
\not = \mu(B)\,\text{and}\,\,
A\cup B=X\textquotedblright \Rightarrow A\cap B=\emptyset}
\]
carries the structure of a cell complex and is the 
{\em Buneman complex} $\comb(\Ss,\ws)$ of $(\Ss,\ws)$. 
Obviously, the set of vertices of $\ch(\Ss,\ws)$ 
consists of those $\mu\in\ch(\Ss,\ws)$ with 
$\mu(A)\in \{0,\ws(\{A,X\setminus A\})\}$ 
for all $A\in \supp(\ws)$. It is easily 
seen that the set $V(\cb(\Ss,\ws))$ of 
vertices of $\cb(\Ss,\ws)$ consists of 
the  $\mu\in\cb(\Ss,\ws)$ with this property.

Setting
\[
d_1(\mu,\nu):=\frac{1}{2} \sum_{A\in \supp(\ws)} |\mu(A)\setminus \nu(A)|
\]
for all $\mu,\nu\in \cb(\Ss,\ws)$, we obtain
a metric space $(\cb(\Ss,\ws),d_1)$ and 
the map $\Phi:X\to \cb(\Ss,\ws)$ defined via
\begin{align}\label{eq:Phi}
\Phi(x)(A):=
\begin{cases}
 \ws(\{A,X\setminus A\}) & \text{if $x\in A$}; \\
 0     & \text{else},
\end{cases}
\end{align}
 for any $x\in X$ and $A\in \supp(\ws)$ 
is an isometric embedding from $(X,d_{(\Ss,\ws)})$ 
to $(\cb(\Ss,\ws),d_1)$.

There exists a natural map 
$\Lambda:\RR^{\supp(\ws)} \to \RR^X,\mu\mapsto f_\mu$ where
\[
f_\mu(x)=\sum_{A\in\supp(\ws)}\mu(A)\quad\text{for all }x\in X\,.
\] It is easily seen that $\Lambda(\Phi(x))=\kappa(x)$ 
for all $x\in X$. Depending on properties of the 
split system $(\Ss,\ws)$, this map takes elements from 
$\ch(\Ss,\ws)$ to elements of $\cp(d_{(\Ss,\ws)})$ 
or even from $\cb(\Ss,\ws)$ to elements of $\ct(d_{(\Ss,\ws)})$; 
see \cite{DHM02,MR1630743,MR1774971} for details. 
In case $(\Ss,\ws)$ 
is weakly compatible and octahedral"=free 
the following holds:

\begin{thm}[Theorem~3.1 in \cite{DHM02}]\label{thm:isom}
If $(\Ss,\ws)$ is a weakly compatible, 
octahedral"=free weighted split system on $X$, 
then the map $\Lambda|_{\cb(\Ss,\ws)}$ is a 
bijection onto $\ct(d_{(\Ss,\ws)})$ that 
induces a cell"=complex isomorphism 
$\Lambda':\comb(\Ss,\ws)\to\comt(d_{(\Ss,\ws)})$.
\end{thm}


Recall that if $\Ss$ is two-compatible, then it is 
weakly compatible and octahedral"=free.
In fact, in the two-compatible case we know even 
more about the relation between $\cb(\Ss,\ws)$ 
and $\ct(d_{(\Ss,\ws)})$.

\begin{thm}
\label{thm:two-decomposable}
Suppose that $(\Ss,\ws)$ is a two-compatible 
weighted split system on~$X$. Then the following hold:
\begin{enumerate}
\item \label{thm:two-decomposable:dim} 
The cell complex $\comb(\Ss,\ws)$
is at most two-dimensional and all two dimensional cells are quadrangles.
\item \label{thm:two-decomposable:isom} The 
map $\Lambda|_{\cb(\Ss,\ws)}:\cb(\Ss,\ws)\to\ct(d_{(\Ss,\ws)})$ 
is an isometry of the metric 
spaces $(\cb(\Ss,\ws),d_1)$ and $(\ct(d_{(\Ss,\ws)}),d_\infty)$.
\end{enumerate}
\end{thm}
\begin{proof}
\noindent (1) By \cite[Lemma~2.1]{DHM01}, the 
dimension of $\cb(\Ss,\ws)$ is bounded by two 
if $(\Ss,\ws)$ is two-compatible. Corollary~7.3 
in \cite{HKM06} states that for weakly 
compatible split systems $(\Ss,\ws)$ all cells in $\comt(d_{(\Ss,\ws)})$ 
are isomorphic to either hypercubes or rhombic dodecahedra. 
Hence (1) follows from Theorem~\ref{thm:isom}.

\noindent (2) This follows from \eqref{thm:two-decomposable:dim} 
in connection with the equivalence of 
(i) and (vi) in \cite[Theorem~1.1~(c)]{DHM01}.
\end{proof}


\subsection{The tight-span}

We have seen in the introduction that any metric space $(X,d)$ 
can be embedded into its tight-span and that 
$d(x,y)=d_\infty(x,y)$ $=d_{(G_d,w_\infty)}(x,y)$ holds for
all $x,y\in X$. A key in the proof of our main 
theorem will be the following observation by Hirai~\cite{Hirai09}, which shows 
that for two-decomposable metrics the latter equality 
holds for general vertices of the tight-span. For the sake of completeness, a proof is included here. Note that the result is not true in general
for metrics with dimension 3 or more~(see, e.g.~\cite[Theorem 3]{koolen2007concerning}). \\

\begin{prop}
~\cite[Proposition 4.2]{Hirai09}
\label{prop:two-decomp:tight-span}
Let $(X,d)$ be a two-decomposable metric space. 
Then for any two elements $f,g\in V(\ct(d))$, we 
have $d_\infty(f,g)=d_{(G_d,w_\infty)}(f,g)$.
\end{prop}

\begin{proof}
For any geodesic $\gamma: [0,1]\to \ct(d)$, let $\Omega(\gamma)$ be the set of all cells $C$ in $\comt(d)$ with dimension two or more such that their intersection with $\gamma$ is not contained in the union of the vertices and edges of $\comt(d)$. By Theorem~\ref{thm:two-decomposable}~\eqref{thm:two-decomposable:dim}, all elements of $\Omega(\gamma)$ are quadrangles. It now suffices to show that for any distinct $f,g\in V(\ct(d))$
there exists a geodesic $\gamma$ between $f$ and $g$ with $\Omega(\gamma)=\emptyset$.

Suppose that is not the case and let $\gamma$ be a geodesic between $f$ and $g$ such that 
the set $\Omega(\gamma)$ has minimal cardinality among all those geodesics. For any $C\in \Omega(\gamma)$, denote the minimal (resp. maximal) element $t\in [0,1]$ with $\gamma(t)\in C$ by $t^-_C$ (resp.~$t^+_C$). Since $\gamma$ is a geodesic, $\gamma(t)\in C$ holds if and only if 
$t\in [t^-_C,t^+_C]$. 

Let $C\in\Omega(\gamma)$ be the first two-dimensional cell met by $\gamma$, that is, the one with minimal $t^-_C$. Clearly, $\gamma(t^-_C)\in V(\ct(d))$ and, since $C$ is a quadrangle and $\gamma(t^+_C)$ belongs to one incident side of $C$, there exists a geodesic segment 
$\gamma'$ between $\gamma(t^-_C)$ and $\gamma(t^+_C)$ using only the boundary edges of $C$. This allows us to construct a geodesic $\gamma^*$ between $f$ and $g$ with
$|\Omega(\gamma^*)|=|\Omega(\gamma)|-1$, a 
contradiction.
\end{proof}


\section{Embeddings in the Buneman complex}
\label{sec:split:flow}

We now start to investigate embeddings of 
optimal realisations into 
the Buneman complex given by Theorem~\ref{thm:inj:geometric:graph}
and show that they have 
the following useful property (where $\Phi$ is the map 
defined 
in Section~\ref{subsect:BC}):

\begin{thm}\label{verth}
Let $(\Ss,\ws)$ be a weighted split system on $X$ and 
$(G,w)$ a minimal path"=saturated optimal realisation 
of $(X,d_{(\ws,\Ss)})$. Then for any non-expansive
map $\psi:||(G,w)|| \to \cb(\Ss,\ws)$ with $\psi(x)=\Phi(x)$ 
for all $x\in X$ we have
$\psi(V(G))\subseteq V(\cb(\Ss,\ws))$.
\end{thm}

Note that this theorem holds for all split system. We shall prove this theorem at the end of this section, after introducing 
and investigating the notions of split-flow 
digraphs and split potentials.\\

We begin with split-flow digraphs. To this end, recalling the following  notation for digraphs. 
A \emph{directed graph} (or \emph{digraph} for short) $\dg=(V,\arc)$ is a pair consisting of a set $V$ of vertices and a  subset $\arc\subseteq V\times V$, whose elements are called the {\em arcs} of $\dg$. A sequence $v_0,v_1,\cdots,v_k$ of vertices of $\dg$ is called a 
{\em directed path} (from $v_0$ to $v_k$) in $\dg$ if $(v_{i-1},v_i)\in \arc$ holds for all 
$i=1,\dots,k$.
A {\em strongly connected component} in $\dg$ is a maximal subset 
$C$ of $V$ such that for all distinct $u,v\in C$ there exists a directed 
path from $u$ to $v$. Obviously, the set of strongly connected components forms a partition of $V$.

Given a realisation $(G,w)$ of a finite metric space $(X,d)$ and some $A\subseteq X$,  the \emph{split-flow digraph} $\dg(G,w;A)$ is the digraph with vertex set $V(G)$ and arc set
\begin{align*}
\arc&:=\big\{(u,v),(v,u)\in V\times V\,\big|\,\text{there exists }x,y\in A \text{ or } x,y\in X\setminus A\\
  &\quad\quad\text{ such that } \{u,v\} \text{ belongs to a shortest path from } x \text{ to }y \text{ in }(G,w)\big\}\\
    &\quad\bigcup\big\{(u,v)\in V\times V\,\big|\,\text{there exists }x\in A \text{ and } y\in X\setminus A\\
  &\quad\quad\text{ such that } \{u,v\} \text{ belongs to a shortest path from } x \text{ to }y \text{ in }(G,w)\big\}\,.
\end{align*}

If $(G,w)$ is an optimal realisation, Lemma~\ref{lem:suff:opt:realisation} implies that for each $\{u,v\}\in E(G)$ we have $(u,v)\in\arc$ or $(v,u)\in\arc$ and both hold if and only if $\{u,v\}$ belongs to a shortest path between two elements of $A$ or between two elements of $X\setminus A$. Furthermore, this implies that there exist strongly connected components $C_A$ and $C_{X\setminus A}$ such that $A\subseteq C_A$ and $X\setminus A\subseteq C_{X\setminus A}$. (Note that $C_A$ and $C_{X\setminus A}$ might be equal.) Note that for a split-flow digraph, there might strongly connected components other than $C_A$ and $C_{X\setminus A}$ (see Fig.~1 
 for an example). But we shall show that for the case of minimal path"=saturated optimal realisations,  these are the only strongly connected components of $\dg(G,w;A)$.


\begin{figure}[h]\label{fig:sfd}
\centering
\includegraphics[height=1.3in]{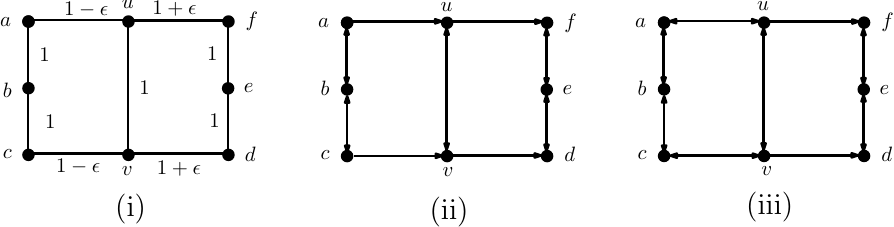}
\caption{
{\small
Examples of split-flow digraphs: (i) An optimal realisation $(G,w_{\epsilon})$ $(-1/2\le \epsilon \le 1/2)$ for the metric in~\cite{Althofer88} on $X=\{a,b,c,d,e,f\}$; (ii) the split-flow digraph $D(G,w_{0};A)$ with $A=\{a,b,c\}$;
 (iii) the split-flow digraph $D(G,w_{1/2};A)$ 
 with $A=\{a,b,c\}$}. Note that the number of shortest paths in $(G,w_{1/2})$ with ends in $X$ is larger than that in $(G,w_{0})$,  and  the split-flow digraph in (ii) has three strongly connected components while that in (iii) has two.
}
\end{figure}

\begin{thm}
\label{thm:split-flow:digraph}
Let $(G,w)$ be a minimal path"=saturated optimal realisation of a metric space $(X,d)$ and $A\subseteq X$. 
Then the strongly connected components of the split-flow digraph $\dg(G,w;A)$ are $C_A$ and $C_{X\setminus A}$. 
In particular, the number of strongly connected components of $\dg(G,w;A)$ is either one or two.
\end{thm}


\begin{proof}
Let $V$ and $\arc$ be the vertex and arc set of $\dg(G,w;A)$, respectively. Assume that there exists a strongly connected component $W$
of $\dg(G,w;A)$ that is 
distinct from both $C_A$ and $C_{X\setminus A}$. It suffices to show that this leads to a contradiction.
 
 \bigskip
The first stage of the proof is to show that
\begin{align*}
\Delta(G,w)=\min\{ w(P)-d(x,y): &~P\not \in \Gamma(G,w;X)~ \mbox{is a path in $(G,w)$ whose } \\ & \mbox{ ends $x$ and $y$ are contained in $X$}\}
\end{align*}
is strictly positive. To see this, consider a vertex $u$ in $W$. Since the degree of $u$ is at least three, applying Lemma~\ref{lem:suff:opt:realisation} 
  to the edges incident with $u$ shows that  $u$ is contained in a path $P_u$ connecting two elements $x_u,y_u\in X$ so that we have either $\{x_u,y_u\} \subseteq A$ or $\{x_u,y_u\} \subseteq  X\setminus A$. As $u$ is contained in neither $C_{A}$ nor $C_{X\setminus A}$, we have $w(P_u)-d(x,y)>0$. Since there are finite many paths in $(G,w)$, it follows that $\Delta(G,w)>0$, as required.\\

 The second stage of the proof is to show that there exists a family of optimal realisations of $(X,d)$ which can be parametrised by an interval and derived from $(G,w)$ by perturbing the weights of a certain set of edges determined by the strongly connected component $W$.   
  To this end, define the map $\delta: E(G) \to \RR$ as
\[
\delta(\{u,v\}):=
\begin{cases}
+1, &\text{if $(u,v)\in \arc$ with $u\in V\setminus W,v\in W$,}\\
-1, &\text{if $(u,v) \in \arc$ with $u\in W, v \in V\setminus W$,}\\
0,  &\text{otherwise.}
\end{cases}
\]
So for any $e\in E(G)$ we have  $\delta(e)=+1$ (resp. $\delta(e)=-1$) if $e$ induces an arc in $\dg(G,w;A)$ that is {\em entering} (resp. {\em leaving}) $W$. By exchanging $A$ and $X\setminus A$ if necessary, we can assume that the difference $N$ of the number of arcs entering and leaving $W$ is non-negative.
Since $W$ is a strongly connected component, for all $u\in V\setminus W$ and $v\in W$ with $\{u,v\}\in E(G)$, we have either $(u,v)\in \arc$ or $(v,u) \in \arc$, but not both. Therefore the above map $\delta$ is well defined. 

Let $t=\min\smallSetOf{w(e)}{e\in E(G) \text{ and } \delta(e)=-1}$. Then we have $t>0$. 
 For any $\epsilon\in (0,t)$,  the map $w_\epsilon:E(G)\to \RR_{>0}$ defined as 
\[
w_\epsilon(e):=w(e)+\epsilon\delta(e)
\]
is a weight function on $E(G)$, and hence $(G,w_\epsilon)$ is a weighted graph. Since $X \cap W=\emptyset$, the numbers of edges entering and leaving $W$ are equal for every path in $\Gamma(\rg,w;X)$. Therefore we have 
\begin{align}
\label{eq:geo:paths}
w(P)=w_\epsilon(P)~~~\mbox{for all}~~P\in \Gamma(\rg,w;X)~~\mbox{and}~~0<\epsilon<t\,,
\end{align} 
and hence also
\begin{align}
\label{eq:path:shorten}
d_{(G,w_\epsilon)}(x,y)\leq d_{(G,w)}(x,y)=d(x,y)\text{ for all }x,y\in X
~~\mbox{and}~~0<\epsilon<t\,.
\end{align}

In addition, we have the following:

\bigskip
\noindent
{\bf Claim-A:} If $(G,w_\epsilon)$ is a realisation of $(X,d)$ for some $0<\epsilon<t$, then $(G,w_\epsilon)$ is an optimal realisation with $\Gamma(G,w;X)\subseteq \Gamma(G,w_\epsilon;X)$. \\

Indeed, we have $l(G,w_{\epsilon})=l(G,w)-\epsilon N$, which, since $(G,w)$ is optimal, implies $N=0$ and that $(G,w_{\epsilon})$ is optimal.  Moreover, by Eq.~(\ref{eq:geo:paths}) we have $\Gamma(G,w;X)\subseteq \Gamma(G,w_\epsilon;X)$, and hence the claim follows.

The last step of the second stage is to show that for all
$$0<\epsilon<\min\Big\{t,\frac{\Delta(G,w)}{|E(G)|}\Big\},$$
$(G,w_\epsilon)$ is an optimal realisation of $(X,d)$ (the existence of such an $\epsilon$ follows from $\Delta(G,w)>0$, as established in the first stage).  To this end, consider two arbitrary elements $x,y$ in $X$ and a path $P$ between them in $G$. If $P$ is contained in $\Gamma(\rg,w;X)$, then we have $w_\epsilon(P)=w(P)\geq d(x,y)$ in view of Eq.~(\ref{eq:geo:paths}). Otherwise, $P$ is not contained in $\Gamma(\rg,w;X)$, and by the definition of $\Delta(G,w)$ we have
\begin{align*}
w_\epsilon(P) &\geq w(P)-|E(P)|\epsilon  \\
&> w(P)-\frac{|E(P)|}{|E(G)|} \Delta(G,w) \\
&\ge w(P)-\Delta(G,w) \\
&\ge d(x,y).
\end{align*}
This implies that for all $x,y\in X$, we have $d_{(G,w_\epsilon)}(x,y)\geq d(x,y)$ and hence $d_{(G,w_\epsilon)}(x,y)= d(x,y)$ in view of Eq.~(\ref{eq:path:shorten}).  Together with {\bf Claim-A}, it follows that $(G,w_\epsilon)$ is an optimal realisation of $(X,d)$, completing the proof of the second stage.


\bigskip

In the final stage of the proof, we shall obtain a contradiction by considering optimal realisations corresponding to the right extremal point in the parameter interval that we found in the second stage of the proof.  To this end, let 
\[
\epsilon'=\sup \SetOf{\epsilon\in (0,t)}{(G,w_\epsilon)~~\text{is a realisation of }(X,d)}\,.
\]
Then we have $\epsilon'>0$ by the second stage.  Now we have the following two cases to consider.\\

\noindent \textbf{Case 1:} $0<\epsilon'<t$. We shall first show that $(G,w_{\epsilon'})$ is an optimal realisation of $(X,d)$. 
By {\bf Claim-A} it suffices to prove that $(G,w_{\epsilon'})$ is a realisation of $(X,d)$.  Indeed, if $(G,w_{\epsilon'})$ is not a realisation, then by Eq.~(\ref{eq:path:shorten}) there exist two elements $x_0,y_0$ in $X$ and a path $P_0$ between $x_0$ and $y_0$ in $G$ such that 
\begin{align*}
w_{\epsilon'}(P_0)=d_{(G,w_{\epsilon'})}(x_0,y_0)<d(x_0,y_0).
\end{align*}
On the other hand, for all $0<\epsilon<\epsilon'$ we have 
$w_{\epsilon}(P_0)\ge d(x_0,y_0)$ because $(G,w_{\epsilon})$ is a realisation of $(X,d)$. Since $w_{\epsilon}(P_0)$ is a continuous function for $0<\epsilon\le \epsilon'$, we have $w_{\epsilon'}(P_0)\ge d(x_0,y_0)$. This is a contradiction, and hence $(G,w_{\epsilon'})$ must be an optimal realisation of $(X,d)$.

Next, fix some $\epsilon_1\in (\epsilon',t)$. Since the weighted graph $(G,w_{\epsilon_1})$ is not a realisation of $(X,d)$, there exists a path $P_1$ with ends $x_1,y_1\in X$ so that $w_{\epsilon_1}(P_1)<d(x_1,y_1)$.
Note that this implies $P_1\not\in \Gamma(\rg,w;X)$ by Eq.~(\ref{eq:geo:paths}).
On the other hand, since $w_{\epsilon}(P_1)\geq d(x_1,y_1)$ for all $0<\epsilon<\epsilon'$ and $w_{\epsilon}(P_1)$ is a continuous function for  $0<\epsilon\le \epsilon'$, we have $w_{\epsilon'}(P_1)=d(x_1,y_1)$, and hence $P_1\in\Gamma(\rg,w_{\epsilon'};X)$. Noting that  $\Gamma(\rg,w;X)\subseteq \Gamma(\rg,w_{\epsilon'};X)$ in view of {\bf Claim-A}, we 
have $|\Gamma(\rg,w_{\epsilon'};X)|> |\Gamma(\rg,w;X)|$,  which is a contradiction to the fact that $(G,w)$ is path"=saturated. This completes the proof of Case 1.\\

\noindent \textbf{Case 2:} $\epsilon'=t$. Let
\[
M:=\SetOf{\{u,v\}\in E(G)}{w(\{u,v\})=t\text{ and }
\delta(\{u,v\})=-1}.
\] 
By definition of $t$, $M$ is not empty.  Next, we will show that $M$ is a matching, that is, no two edges in $M$ share a common vertex. Suppose $M$ contains $e_1=\{u,v\}$ and $e_2=\{v,u'\}$ for some $u,v,u'\in V(G)$.  We may assume $v\in W$ and $u,u'\in V\setminus W$, as the other case (i.e., $u,u'\in W$ and $v\in V\setminus W$) is similar. By Lemma~\ref{lem:suff:opt:realisation}~\eqref{lem:suff:opt:realisation:two}, this implies that $e_1$ and $e_2$ are contained in a shortest path between two elements of~$X$. By the construction of $\dg(G,w;A)$, this is a contradiction  to the assumption that $W$ is a strongly connected component.

Now let $G'$ be the graph obtained from $G$ by \emph{contracting} all edges in $M$. That is, for every edge $e=\{u,v\}$ in $M$, where $u\in W$, we add an edge $\{u',v\}$ for all $u'\not=v$ that is adjacent to $u$, and delete $u$ and all edges incident to it. Since 
$M$ is a matching, the graph $G'$ is  well-defined. We define the weight function $w':E(G')\to \RR_{>0}$ by setting $w'(e)=w(e)+t\delta(e)$ if $e\in E(G)$ and $w'(e)=w(e^*)+t\delta(e^*)$ otherwise, where $e^*=\{u',u\}$ for the unique vertex $u$ in $G$ so that $\{u,v\} \in M$.

Since contracting length-zero edges does not change any path lengths, an argument similar to the one in Case 1 shows that $(G',w')$ is an optimal realisation of $(X,d)$.


Now consider two arbitrary paths $P_1,P_2\in\Gamma(G,w;X)$ and let $P'_1,P'_2\in\Gamma(G',w';X)$ be the paths obtained by contracting all its edges contained in $M$. Then it remains to show that $P'_1\not =P'_2$, because this implies that $|\Gamma(G',w';X)|\geq|\Gamma(G,w;X)|$,  and hence $(G',w')$ is a path"=saturated optimal realisation of $(X,d)$ with $|V(G')|<|V(G)|$, a contradiction as required. 
Indeed, if $P'_1=P'_2$, then there exists some edge $e=\{u,v\}\in M$ such that $e$ is contained in exactly one path among $P_1,P_2$, say~$P_1$. Switching the role of $u$ and $v$ if necessarily, we have $(u,v)\in \arc$. Let $s$ be the other vertex in $P_1$ that is adjacent to $u$. Since no edges in $M$ share a common vertex, we have $\{s,u\}\not \in M$ and hence $\{s,v\}\in P'_1=P'_2$. Again using the fact that no two edges in $M$ share a common vertex, we get $\{s,v\}\in P_2$ and so $\{u,v\},\{s,v\},\{s,u\}\in E(G)$, contradicting Corollary~\ref{cor:triangle:free}. Therefore we have $P'_1\not = P'_2$, which completes the proof of Case 2, as well as the theorem.
\end{proof}


Let $(G,w)$ be a realisation of $(X,d)$. A
map $\lambda: ||(G,w)|| \to [0,1]$ is called 
a \emph{split potential} on $||(G,w)||$ if
\begin{enumerate}
\item $\lambda(x) \in \{0,1\}$ for all $x \in X$, and 
\item $\lambda\circ\gamma:[0,1]\to [0,1]$ is monotonic for all $X$-geodesics $\gamma:[0,1] \to ||(\rg,w)||$.
\end{enumerate}

\begin{lem}\label{potential}
Let $(G,w)$ be a minimal path"=saturated optimal realisation of  $(X,d)$ and $\lambda$ a split potential on $||(G,w)||$. Then $\lambda(v) \in \{0,1\}$ holds for all $v \in  V(G)$.
\end{lem}

\begin{proof}
Let $A$ be the set of all $x\in X$ that are mapped to $0$ by $\lambda$ and let $\dg=(V,\arc)$ be the split-flow digraph $\dg(\rg,w;A)$. 

We will now show that $(u,v)\in \arc$ implies that $\lambda(u)\leq \lambda(v)$. This implies that $\lambda$ restricted to any strongly connected component of $\dg$ is a constant, and the lemma then follows from Theorem~\ref{thm:split-flow:digraph}. Indeed, fix $a\in A$ and $b\in X-A$; then we have $\{f(a),f(b)\}\subseteq \{0,1\}$ and Theorem~\ref{thm:split-flow:digraph} implies that we have $f(v)\in \{f(a),f(b)\}$  for each $v\in V(G)$.

So let $(u,v)\in \arc$. If $u\in A$ or $v\in X\setminus A$, then $\lambda(u)\leq \lambda(v)$ obviously holds. Otherwise, by Lemma~\ref{lem:suff:opt:realisation}(i), there exists $x,y \in X$ and a shortest path $P$ in $(G,w)$ from $x$ to $y$ such that $\{u,v\}$ is an edge in $P$ and either $x\in A$ or $y\in X\setminus A$. Here we shall consider the case $x\in A$ since the other one is similar. The path $P$ induces an $X$-geodesic $\gamma$ between $x$ and $y$ that passes first through $u$ and then through $v$. Since $\lambda\circ \gamma$ is monotonic and $\lambda(\gamma(0))=\lambda(x)=0<1=\lambda(y)=\lambda(\gamma(1))$, we get $\lambda(u)\leq\lambda(v)$, as required.
\end{proof}

We now present a specific way to define 
split potentials using the Buneman complex,
which will allow us to prove Theorem~\ref{verth}. 

Let $(\Ss,\ws)$ be a weighted split system on $X$.
For any $A\in \supp(\ws)$, define the 
map $\lambda_A: \cb(\Ss,\ws) \rightarrow [0,1]$
by putting
\[
\lambda_A(\mu) := \frac{\mu(A)}{\ws(\{A,X\setminus A\})}
\quad\text{for all }\mu \in \cb(\Ss,\ws)\,.
\]

\begin{lem}\label{ham:split-potential}
Let $(\Ss,\ws)$ be a weighted split system on $X$, $A\in \supp(\ws)$ and $(G,w)$ a realisation of $(X,d_{(\ws,\Ss)})$. Then for any non-expansive map $\psi:||(G,w)|| \to \cb(\Ss,\ws)$ with $\psi(x)=\Phi(x)$ for all $x\in X$, the function $\lambda_A\circ \psi$ is a split potential on $||(G,w)||$.
\end{lem}

\begin{proof}
Let $\gamma: [0,1]\rightarrow ||(\rg,w)||$ be an $X$-geodesic in $||(G,w)||$ and set $\lambda'_A:=\lambda_A \circ \psi \circ \gamma$. Since  $\psi(x)=\Phi(x)$, we have $(\lambda_A\circ \psi)(x) \in \{0,1\}$ for all $x\in X$, so it remains to show that $\lambda'_A$ is monotonic.

Since $\psi$ is non-expansive,
the map $\psi\circ \gamma:[0,1]\to \cb(\Ss,\ws)$
is a geodesic in $(\cb(\Ss,\ws),d_1)$. Therefore, for any
 $0\leq t_1<t_2<t_3\leq 1$, we have
\[
d_1(\nu_1,\nu_3)=
d_1(\nu_1,\nu_2)+d_1(\nu_2,\nu_3)\,,
\]
where $\nu_i:=\psi\circ \gamma(t_i)$ for $i=1,2,3$. By the definition of the metric $d_1$, this gives us
\[
\sum_{B \in \supp(\ws)} |\nu_1 (B)-\nu_3 (B)| =
\sum_{B \in \supp(\ws)}\big(|\nu_1 (B)-\nu_2 (B)|+ |\nu_2 (B)-\nu_3 (B)| \big).
\]
Since, by definition, for all $B\in \supp(\ws)$, we have 
\[
|\nu_1(B)-\nu_3(B)|\leq|\nu_1(B)-\nu_2(B)|+|\nu_2(B)-\nu_3(B)|\,,
\]
this implies that 
\[
|\nu_1 (A)-\nu_3 (A)|= |\nu_1 (A)-\nu_2 (A)|+|\nu_2 (A)-\nu_3 (A)|\,,
\]
and hence 
\[
\big(\nu_1 (A)-\nu_2 (A)\big)\cdot \big(\nu_2 (A)-\nu_3 (A)\big) \geq 0.
\]
\noindent
Together with  
$$
\lambda'_A(t_i)=\lambda_A(\nu_i)=\frac{\nu_i(A)}{\alpha(\{A,X\setminus A\})}
$$
 for $i=1,2,3$, this implies
\[
\big(\lambda'_A(t_1)-\lambda'_A(t_2)\big)\cdot \big(\lambda'_A(t_2)-\lambda'_A(t_3) \big) \geq 0\,,
\]
hence $\lambda'_A$ is monotonic.
\end{proof}

With the above results, we are now in a position to present the proof of
the theorem stated at the beginning of this section. 

\begin{proof}[Proof of Theorem~\ref{verth}]
Lemmas \ref{potential} and \ref{ham:split-potential} show that for all $A\in\supp(\ws)$ and $v\in V(G)$ we have $\lambda_A(\psi(v))\in\{0,1\}$. However, by the definition of the map $\lambda_A$, we have $\lambda_A(\mu)\in\{0,1\}$ if and only if $\mu(A)\in\{0,\ws(\{A,X\setminus A\})\}$ for all $\mu\in \cb(\Ss,\ws)$. So $\psi(v)$ is a vertex of $\cb(\Ss,\ws)$.
\end{proof}


\section{The main theorem} 
\label{sec:proof}

We are now ready to prove the main result of this paper,
from which Theorem~\ref{2compat} follows immediately.

\begin{thm}\label{thm:min-path-sat}
Let $(X,d)$ be a totally"=decomposable finite metric space with dimension two and $(G,w)$ a minimal path"=saturated optimal realisation of $(X,d)$. Then $(G,w)$ is homeomorphic to a subgraph of $(G_d,w_\infty)$. 
\end{thm}

\begin{proof}
Since $(G,w)$ is path-saturated, by Theorem~\ref{thm:inj:geometric:graph} there exists a non-expansive injection $\psi$ from $||(\rg,w)||$ to  $\ct(d)$ satisfying $\psi(x)=\kappa(x)$ for all $x\in X$, where $\kappa$ is the embedding of $(X,d)$ into its tight-span.
\bigskip
The first stage of the proof is to show that the injection $\psi$ maps vertices of $G$ to vertices of $\ct(d)$, that is, 
\begin{align}
\label{eq:inc:vertex}
\psi(V)\subseteq V(\ct(d))=V(G_d). 
\end{align}
To see this, consider the weighted two-compatible split system $(\Ss,\ws)$  on $X$ such that $d=d_{(\Ss,\ws)}$.
Then $(X,d)$ is embedded into the Buneman complex $\cb(\Ss,\ws)$ via the map $\Phi$ (see Section~\ref{subsect:BC}). By Theorem~\ref{thm:two-decomposable}~\eqref{thm:two-decomposable:isom}, there exists an isometry $\Lambda$ from $\cb(\Ss,\ws)$ to $\ct(d)$ such that $\Lambda(\Phi(x))=\kappa(x)$ holds for all $x\in X$. Noting that the map $\phi$ is non-expansive, it follows that the map $\psi':=\Lambda^{-1}\circ \psi$ is a non-expansive map from $||(\rg,w)||$ to $\cb(\Ss,\ws)$ with $\psi'(x)=\Phi(x)$ for all $x\in X$. By Theorem~\ref{verth} we have $\psi'(V)\subseteq V(\cb(\Ss,\ws))$, and hence $\Lambda(\psi'(V))=\psi(V)\subseteq V(\ct(d))=V(G_d)$. Hence Eq.~(\ref{eq:inc:vertex}) holds.

\bigskip

The second stage of the proof is to construct an optimal realisation of 
$(X,d)$ that is a subgraph of $(G_d,w_\infty)$.

To simplify the notation, we set $\ci{u}:=\psi(u)$ for all $u\in V(G)$ (note that $\ci{x}=x$ for all $x\in X$). Now, for each edge $e=\{u,v\}\in E(G)$,
since $\ci{u}$ and $\ci{v}$ are vertices in $G_d$ by Eq.~(\ref{eq:inc:vertex}), 
 we fix a shortest path $P_e$ in $G_d$ between $\ci{u}$ 
and $\ci{v}$ and let $\mathcal{P}:= \smallSetOf{P_e}{e\in E(G)}$ be the collection of all such paths.
We now consider the subgraph $(G^*,w^*)$ of $(G_d,w_\infty)$ defined by\begin{align*}
V(G^*)&=\SetOf{u\in V(G_d)}{u\text{ is contained in some path }P\in \mathcal{P}}\,,\\
E(G^*)&=\SetOf{\{u,v\}\in E(G_d)}{\text{$u$ and $v$ are adjacent in some path }P\in \mathcal{P}}
\end{align*}
and $w^*=w_\infty|_{E(G^*)}$. The aim of the second stage is to show that $(G^*,w^*)$ is an optimal realisation of $(X,d)$.

Firstly, note that there exists a map $\tau$ associating each path $P=v_0,v_1,\dots,v_k$ ($k\ge 1$) in  $\Gamma(G,w;X)$ with the walk\footnote{A \emph{walk} in $G^*$ is a sequence of not necessarily distinct vertices in $G^*$ with each consecutive pair forming an edge in $G^*$;   we shall show later in the proof that $\tau(P)$ is actually a path in $\Gamma(G^*,w^*;X)$.}  $\tau(P)$ in $(G^*,w^*)$ between $\ci{v}_0$ to $\ci{v}_k$ that is obtained from $P$ by replacing each edge $e_i=\{v_{i-1},v_i\}$ ($1\le i \le k$) with the path $P_{e_i}$ in $\mathcal{P}$.
In particular, if $P$ contains only one edge $e$, then $\tau(P)=P_e$.

Secondly,  we have 
\begin{align}\label{eq:length}
w(e)=d_{(\rg,w)}(u,v) \geq d_\infty(\ci{u},\ci{v})=
d_{(G_d,w_\infty)}(\ci{u},\ci{v})=w_\infty(P_e)=w^*(P_e)
\end{align}
for all edge $e=\{u,v\}$ in $E(G)$, where the inequality follows from $\psi$ being non-expansive, and the second equality follows from Proposition~\ref{prop:two-decomp:tight-span}.

Moreover, we claim that for each pair of distinct elements $x,y\in X$, we have 
\begin{align}
\label{eq:length:star}
d_{(G^*,w^*)}(x,y)\leq d_{(\rg,w)}(x,y).
\end{align}
To see this, fix a shortest path $P_{x,y}$ in $G$ (and hence $P_{x,y}\in \Gamma(G,w,;X)$) and consider the walk $\tau(P_{x,y})$ between $x$ and $y$ in $(G^*,w^*)$. Then  we have
\begin{align*}
d_{(G^*,w^*)}(x,y) \le w^*(\tau(P_{x,y})) =\sum_{i=1}^k w^*(P_{e_i})
\leq \sum_{i=1}^k w(e_i)=w(P_{x,y})=d_{(G,w)}(x,y),
\end{align*}
where the second inequality follows from Eq.~(\ref{eq:length}). Hence Eq.~(\ref{eq:length:star}) holds as claimed.

Next, since $(G_d,w_\infty)$ and $(G,w)$ are both realisations of $(X,d)$, by Eq.~(\ref{eq:length:star}) for all $x,y\in X$ we have
\[
d(x,y)=d_{(G_d,w_\infty)}(x,y)\leq d_{(G^*,w^*)}(x,y)\leq d_{(\rg,w)}(x,y)=d(x,y)\,.
\]
Hence 
\begin{align}
d(x,y)=d_{(G_d,w_\infty)}(x,y)= d_{(G^*,w^*)}(x,y)= d_{(\rg,w)}(x,y)=w^*(\tau(P_{x,y}))\,
\end{align}
holds for every shortest path $P_{x,y}$ between $x$ and $y$ in $G$.
Therefore $(G^*,w^*)$ is a realisation of $(X,d)$. In addition, this implies $\tau(P_{x,y})$ is a shortest path between $x$ and $y$ in $(G^*,w^*)$, and hence $\tau$ is a map from $\Gamma(G,w;X)$ to $\Gamma(G^*,w^*;X)$.

Finally, 
we have
\begin{align*}
l(G^*,w^*)\leq \sum_{P\in \cP}w_\infty(P)
=\sum_{e\in E(G)}w_\infty(P_e)
\leq \sum_{e\in E(G)}w(e)=l(G,w)\,,
\end{align*}
where the the first inequality follows from construction and the second one from Eq.~(\ref{eq:length}). Since $(G,w)$ is optimal, this implies that $(G^*,w^*)$ is also optimal, and that
\begin{align}
\label{eq:optimal:length}
l(G^*,w^*)= \sum_{P\in \cP}w_\infty(P)
=\sum_{e\in E(G)}w_\infty(P_e) =
 \sum_{e\in E(G)}w(e)=l(G,w).
\end{align}
This complete the second stage of the proof.


\bigskip
 In the third and final stage of the proof, we shall  show that $(G^*,w^*)$  is homeomorphic to $(G,w)$.  Suppose first that for any two distinct paths $P_{e_1}$ and $P_{e_2}$ in $\mathcal{P}$ with $e_i=(u_i,v_i)$ for $i=1,2$, 
we have 
\begin{equation}
\label{eq:disjoint:paths}
V(P_{e_1})\cap V(P_{e_2}) \subseteq \{\ci{u}_1,\ci{v}_1\} \cap \{ \ci{u}_2,\ci{v}_2\}.
\end{equation}
Then the weighted graph 
obtained from $(G^*,w^*)$ by suppressing all vertices with degree two in $V(G^*)\setminus \psi(V)$ is isomorphic to $(G,w)$. That is, $(G,w)$ is homeomorphic to $(G^*,w^*)$, a subgraph of $(G_d,w_\infty)$, as desired. So it remains to show that Inclusion~\eqref{eq:disjoint:paths} always holds.

To this end, note first that two distinct paths in $\mathcal{P}$ do not share a common edge in $(G_d,w_\infty)$, because otherwise we have $l(G^*,w^*)< \sum_{P\in \cP}w_\infty(P)$, a contradiction to Eq.~(\ref{eq:optimal:length}).
Therefore, the map $\tau: \Gamma(\rg,w,;X) \to \Gamma(\rg^*,w^*;X)$ is injective. 

Secondly, we must have  $|V(P_{e_1})\cap V(P_{e_2})|<2$. Indeed, if this were not the case, there would exist two vertices $u$ and $v$ in $V(P_{e_1})\cap V(P_{e_2})$. Let $P_1$ and $P_2$ denote the subpath from $u$ to $v$ induced by $P_{e_1}$ and $P_{e_2}$, respectively.  Now consider the 
path $P^{\circ}_{e_2}$ obtained from $P_{e_2}$ by replacing $P_2$ by $P_1$ and let $\mathcal{P}^{\circ}$ be the collection of paths obtained from $\mathcal{P}$ by replacing $P_{e_2}$ with ${P}^{\circ}_{e_2}$. Since two distinct paths in $\mathcal{P}$ do not share a common edge, we have 
\begin{align}
\label{eq:alt:length}
\sum_{P\in {\cP}^{\circ} }w_\infty(P)< \sum_{P\in \cP}w_\infty(P)=l(G^*,w^*). 
\end{align}
 On the other hand, using an argument similar to showing that $(G^*,w^*)$ is an optimal realisation of $(X,d)$, we know that the graph $(G^{\circ},w^{\circ})$ 
obtained as the union of the paths in $\mathcal{P}^{\circ}$ is also an optimal realisation of $(X,d)$, a contradiction to Eq.~(\ref{eq:alt:length}) and the fact that $(G^*,w^*)$ is optimal. Hence  $|V(P_{e_1})\cap V(P_{e_2})|<2$ as claimed.


So, to complete the proof, assume that $(V(P_{e_1})\cap V(P_{e_2}))\setminus( \{\ci{u}_1,\ci{v}_1\}\cap \{\ci{u}_2,\ci{v}_2\})=\{v\}$ for some $v\in V(G)$. For $i=1,2$, let $f_i$ and $f'_i$
be the two edges in $P_{e_i}$ that 
are incident with $v$. Note that $\{f_1,f'_1\}\cap \{f_2,f'_2\}=\emptyset$. Since $(\rg^*,w^*)$ is isomorphic to an optimal realisation,  
Lemma~\ref{lem:suff:opt:realisation} implies that 
there exists a path $P^*$ in $\Gamma(\rg^*,w^*;X)$ 
that contains the edges $f_1$ and $f_2$.
Since for each path $P$ in $\Gamma(\rg,w;X)$, the path $\tau(P)$ in $(\rg^*,w^*)$ can contain at most two edges incident with $v$ and the set of  these two edges must be $\{f'_1,f_1\}$ or $\{f'_2,f_2\}$ (but not both), we know that $P^*$ is not the image of any path $P$ in $\Gamma(\rg,w;X)$ under the map $\tau$. Therefore, the map $\tau: \Gamma(\rg,w,;X) \to \Gamma(\rg^*,w^*;X)$  is not surjective. But $\tau$ is injective, and so    
 $|\Gamma(\rg^*,w^*;X)|>|\Gamma(\rg,w;X)|$, 
a contradiction to the assumption that $(\rg,w)$ is path"=saturated.
This completes the proof of Eq.~\eqref{eq:disjoint:paths}, and hence also the theorem.\end{proof}

\subsection*{Acknowledgements}
Work of SH was supported by a 
fellowship within the Postdoc"=Programme of 
the German Academic Exchange Service (DAAD) and
the University of East Anglia. Part of the work of TW was supported by the Singapore MOE grant
R-146-000-134-112. JHK was partially supported 
by the Basic Science Research Program through 
the National Research Foundation of Korea (NRF)
funded by the Ministry of Education, Science and Technology 
(2010-0008138), and JHK was also  partially supported by the `100 talents' program of the Chinese Academy of
Sciences. Finally, we thank the comments and suggestions of two anonymous reviewers that have led to a substantial improvement of this paper.

\bibliographystyle{amsplain}
\bibliography{2-decomp}

\end{document}